\documentclass[a4paper,DIV=10,twopage]{scrartcl}

\usepackage[utf8]{inputenc}

\usepackage{amsmath}
\usepackage{amsthm}
\usepackage{amsfonts}
\usepackage{amssymb}
\usepackage{mathtools}

\usepackage{graphicx}
\usepackage{caption}

\usepackage{tikz}
\usetikzlibrary{calc,cd,quotes}

\def\N{\mathbb{N}}
\def\Z{\mathbb{Z}}
\def\Q{\mathbb{Q}}
\def\R{\mathbb{R}}

\newcommand{\Rnn}{\R_{\geq 0}}

\newcommand{\sDMO}[1]{\expandafter\DeclareMathOperator\csname#1\endcsname{#1}}
\sDMO{id}
\sDMO{Id}
\sDMO{Ob}
\sDMO{sk}
\sDMO{map}

\newcommand{\MakeCategoryName}[1]{%
    \expandafter\DeclareMathOperator\csname#1\endcsname{\mathsf{#1}}
}
\MakeCategoryName{Top}
\MakeCategoryName{Vect}
\MakeCategoryName{snVect}

\newcommand{\KVectc}{\mathop{\Vect_K^\omega}}
\newcommand{\QVectfin}{\mathop{\Vect_\Q^{\fin}}}
\newcommand{\KVectfin}{\mathop{\Vect_K^{\fin}}}
\MakeCategoryName{Fsn}

\newcommand{\htpycat}[1]{#1_{\mathrm{h}}}

\newcommand{\fin}{\mathrm{fin}}
\newcommand{\inv}{^{-1}}
\newcommand{\setZeroTo}[1]{\{0,\ldots,#1\}}
\newcommand{\setOneTo}[1]{\{1,\ldots,#1\}}
\newcommand{\oneInF}[1]{1_{#1}}
\def\fa#1{\forall_{#1}\quad}
\def\args{\,\cdot\,}
\newcommand{\after}{\circ}
\newcommand{\isom}{\cong}
\newcommand{\natTo}{\Rightarrow}

\usepackage[pagebackref,
    ,pdfborder={0 0 0}
    ,urlcolor=black,a4paper,hypertexnames=false
    ,pdftitle={Universal finite functorial semi-norms}]{hyperref}
\hypersetup{pdfauthor={Clara L\"oh, Johannes Witzig}}

\usepackage[markcase=ignoreuppercase]{scrlayer-scrpage}

\automark[section]{section}
\usepackage{ifthen}
\deftriplepagestyle{myheadings}
             {\ifthenelse{\isodd{\value{page}}}{\leftmark}{\pagemark}}
             {}
             {\ifthenelse{\isodd{\value{page}}}{\pagemark}{\rightmark}}%
             {}{}{}%
\newtheoremstyle{ggt}{}{}{\itshape}{}{\sffamily\bfseries}{.}{ }{}
\newtheoremstyle{ggtdefinition}{}{}{}{}{\sffamily\bfseries}{.}{ }{}

\theoremstyle{ggt}
\newtheorem{thm}{Theorem}[section]
\newtheorem{lemma}[thm]{Lemma}
\newtheorem{cor}[thm]{Corollary}
\newtheorem{prop}[thm]{Proposition}

\theoremstyle{ggtdefinition}
\newtheorem{defi}[thm]{Definition}
\newtheorem{example}[thm]{Example}
\newtheorem{question}[thm]{Question}
\newtheorem{remark}[thm]{Remark}
\newtheorem{setup}[thm]{Setup}

\makeatletter
\let\origthmhead=\thmhead
\renewcommand{\thmhead}[3]{%
\origthmhead{#1}{#2}{#3}%
\belowpdfbookmark{#1\@ifnotempty{#1}{ }#2\thmnote{ (#3)}}{#1#2}%
}
\makeatother

\newcommand{\lone}{\texorpdfstring{$\ell^1$}{l1}}

\newcommand{\pref}[1]{(\ref{#1})}

\newcommand{\subref}[2]{\ref{#1}\,\pref{#1#2}}

\mathcode`\*=\cdot

\makeatletter
\def\blfootnote{\xdef\@thefnmark{}\@footnotetext}
\makeatother
\def\draftinfo{}
\date{\today%
  \protect\blfootnote{\copyright{\ C.~L\"oh, J.~Witzig~2022}. 
    This work was supported by the CRC~1085 \emph{Higher Invariants} 
    (Universit\"at Regensburg, funded by the DFG).
    \draftinfo}}

\begin{document}

\title{Universal finite functorial semi-norms}
\author{Clara L\"oh, Johannes Witzig}

\maketitle

\thispagestyle{empty}

\begin{abstract}
  Functorial semi-norms on singular homology measure the ``size'' of
  homology classes. A geometrically meaningful example is the
  $\ell^1$-semi-norm.  However, the $\ell^1$-semi-norm is not
  universal in the sense that it does not vanish on as few classes as
  possible.  We show that universal finite functorial semi-norms do
  exist on singular homology on the category of topological spaces
  that are homotopy equivalent to finite CW-complexes. Our arguments
  also apply to more general settings of functorial semi-norms.
\end{abstract}

\section{Introduction}

A functorial semi-norm on a functor~$F \colon C \to \Vect_K$ to vector
spaces over a normed field~$K$ is a lift of~$F$ to a functor~$C \to
\snVect_K$ to the category of semi-normed vector spaces over~$K$
(Definition~\ref{def:ffsn}). A functorial semi-norm on~$F$ is called
universal if it vanishes on as few classes as possible among all
functorial semi-norms on~$F$ (Definition~\ref{def:uffsn}).

A geometrically meaningful example of a functorial semi-norm is the
$\ell^1$-semi-norm on singular homology~\cite{gromov82}, which
measures the ``size'' of homology classes in terms of singular
simplices and has applications to rigidity of
manifolds~\cite{gromov82,bcg,lafontschmidt,connellwang}. It is known
that the $\ell^1$-semi-norm is \emph{not} universal in high
degrees~\cite{fauserloeh19} (see also Example~\ref{exa:nonuniversal})
and it is thus natural to ask whether
universal finite functorial semi-norms exist on singular
homology~\cite[Question~4.2]{fauserloeh19}. In the present article, we
answer this question affirmatively on the category of spaces homotopy
equivalent to finite CW-complexes (Corollary~\ref{cor:singhomfin}).

More generally, using a suitable diagonalisation technique, we prove
the following general existence result (Section~\ref{sec:exproof}):

\begin{thm}\label{thm:genuffsn}
  Let $C$ be a category that admits a skeleton with
  at most countably many objects. Let $K$ be a normed
  field and let $F \colon C \to \Vect_K$ be a functor.
  \begin{enumerate}
  \item
    \label{thm:genuffsn:countable}
    If $K$ is countable and if $F$ maps to~$\KVectc$,
    then $F$ admits a universal finite functorial semi-norm.
  \item
    \label{thm:genuffsn:fin}
    If $F$ maps to~$\KVectfin$, then $F$ admits a universal
    finite functorial semi-norm.
  \end{enumerate}
\end{thm}

Here, $\KVectfin$ and $\KVectc$ denote the categories of $K$-vector
spaces of finite or countable dimension, respectively.  
If the countability assumption on the skeleton is dropped, then in
general there does not need to exist a universal finite functorial
semi-norm (Section~\ref{sec:nonuffsn}).

Instantiating Theorem~\subref{thm:genuffsn}{:fin} to singular homology, we
obtain (Section~\ref{subsec:proofsinghomfin}):

\begin{cor}\label{cor:singhomfin}
  Let $d \in \N$ and let $K$ be a normed field (e.g., $\R$).  Then the
  singular homology functor~$H_d(\args;K)$ admits a universal finite
  functorial semi-norm on the category of all topological spaces that
  are homotopy equivalent to finite CW-complexes.
\end{cor}

In degrees~$d \in \{0,1\}$, it is easy to determine explicit universal
finite functorial semi-norms on~$H_d(\args;\R)$
(Example~\ref{exa:deg01}). However, the following problems remain
open:

\begin{question}\label{q:explicit}
  What is the geometric meaning of universal finite functorial
  semi-norms on singular homology? Are there ``nice'' examples, at
  least in degrees~$2$ and~$3$?
\end{question}

We reformulate Question~\ref{q:explicit} in more concrete terms
in Remark~\ref{rem:URCgrowth}.

\begin{question}
  Let $d \in \N_{\geq 2}$. Does Corollary~\ref{cor:singhomfin} also
  hold for singular homology on the category of \emph{all} topological
  spaces?
\end{question}

\begin{remark}[a comment on sets]
  As underlying set theory, we use NBG-style sets and classes; this
  leads to smallness assumptions in some places. Of course, similarly,
  one could also work in other types of foundations.
\end{remark}
  
\subsection*{Organisation of this article}

We start by recalling the notion of (universal) finite functorial
semi-norms as well as basic examples and constructions in
Section~\ref{sec:ffsn}.  In Section~\ref{sec:equiv}, we show that
universality is compatible with equivalences of categories.  The key
construction for universality is presented in
Section~\ref{sec:vanishingloci}, which allows us to prove the
existence results in Theorem~\ref{thm:genuffsn} and
Corollary~\ref{cor:singhomfin} in Section~\ref{sec:exproof}. Moreover,
Section~\ref{sec:nonuffsn} contains an example of a functor that does
not admit a universal finite functorial semi-norm.

\section{Finite functorial semi-norms}\label{sec:ffsn}

We recall basic notions and examples for finite functorial semi-norms,
with a focus on the case of singular homology.

We use the following terminology: Let $K$ be a normed field (e.g., $\Q$
or~$\R$ with the standard norm).
A \emph{semi-norm} on a $K$-vector space~$V$ is a
function~$|\cdot|\colon V\to\R_{\geq 0}\cup\{\infty\}$ that satisfies
\begin{itemize}
    \item $|0| = 0$, the
    \item triangle-inequality, i.e., for all $x,y\in V$ we have
        $|x+y| \leq |x| + |y|$, and 
    \item homogeneity, i.e., for all $a\in K\setminus\{0\}$
        and all $x\in V$ we have $|a * x| = |a| * |x|$
\end{itemize}
(with the usual conventions regarding~$\infty$). A semi-norm is
\emph{finite} if $\infty$ is not attained.
We denote the category of $K$-vector spaces by~$\Vect_K$ and
the category of semi-normed $K$-vector spaces with norm non-increasing
$K$-homomorphisms by~$\snVect_K$.

\begin{setup}\label{setup:F}
    Let $C$ be a category, let $K$~be a normed field, and let
    $F\colon C\to\Vect_K$ be a functor.
\end{setup}

\subsection{Functorial semi-norms}

\begin{defi}[\texorpdfstring{$F$}{F}-element]
    In the situation of Setup~\ref{setup:F},
    an \emph{$F$-element} is a pair $(X,\alpha)$ where
    $X\in\Ob(C)$ and $\alpha\in F(X)$.
    We often suppress~$X$ in the notation and simply say
    that \emph{$\alpha$~is an $F$-element}.
\end{defi}

\begin{defi}[(finite) functorial semi-norm]\label{def:ffsn}%
    We consider the situation of Setup~\ref{setup:F}.
    A \emph{functorial semi-norm on~$F$} is a
    lift of~$F$ to a functor $\sigma\colon C\to\snVect_K$.
    Explicitly, the latter consists of a semi-norm~$|\cdot|_\sigma$
    on~$F(X)$ for all objects~$X$ of~$C$, such that for all
    morphisms~$f\colon X\to Y$ of~$C$ and all $\alpha\in F(X)$ we have
    \[ \bigl|F(f)(\alpha)\bigr|_\sigma \leq |\alpha|_\sigma  . \]
    A functorial semi-norm on~$F$ is \emph{finite} if $|\cdot|_\sigma$ is
    finite on~$F(X)$ for all~$X$.
\end{defi}

\begin{example}[trivial functorial semi-norm]
    Every functor in Setup~\ref{setup:F} admits the \emph{trivial
      functorial semi-norm}, i.e., the semi-norm that vanishes on
    every input.
\end{example}

\begin{defi}[carries] \label{def:carries}
    In the situation of Setup~\ref{setup:F}, let $\sigma$ and~$\tau$
    be functorial semi-norms on~$F$.
    Then \emph{$\sigma$~carries~$\tau$} if for all
    $F$-elements~$\alpha$, we have
    \[ |\alpha|_\sigma = 0 \implies |\alpha|_\tau = 0  . \]
\end{defi}

\begin{defi}[universal finite functorial semi-norm]\label{def:uffsn}
    In the situation of Setup~\ref{setup:F}, a \emph{universal
    finite functorial semi-norm on~$F$} is a finite functorial
    semi-norm on~$F$ that carries all other finite functorial
    semi-norms on~$F$.
\end{defi}
\begin{remark}
    Definition~\ref{def:uffsn} is not interesting for the
    non-finite case, because the functorial semi-norm that is~$\infty$
    everywhere, except at~$0$, is always universal.
\end{remark}

\begin{example}[\lone-semi-norm]
    Let $d\in\N$. For a topological space~$X$, we set
    \[ \Bigl| \sum_{j=1}^N a_j * \sigma_j \Bigr|_1
        := \sum_{j=1}^N |a_j|
    \]
    for all reduced singular chains~$\sum_{j=1}^N a_j * \sigma_j$
    in~$C_d(X;\R)$. The norm~$|\cdot|_1$ on~$C_d(X;\R)$ induces a finite
    semi-norm~$\|\cdot\|_1$ on singular homology~$H_d(X;\R)$ via
    \[ \|\alpha\|_1 :=
        \inf\bigl\{ |c|_1 \bigm| c\in C_d(X;\R) \text{ is a cycle
        representing } \alpha \bigr\}
    , \]
    which is easily seen to be functorial in the sense of
    Definition~\ref{def:ffsn}. Hence, we obtain the
    \emph{$\ell^1$-semi-norm~$\|\cdot\|_1$} on~$H_d(\;\cdot\;;\R)$.
    
    An invariant defined in terms of the $\ell^1$-semi-norm is the
    \emph{simplicial volume}, introduced by Gromov~\cite{gromov82}:
    For an oriented closed connected $d$-dimensional manifold~$M$, the
    \emph{simplicial volume}~$\|M\|$ of~$M$ is the
    $\ell^1$-semi-norm~$\|M\| := \| [M]_\R \|_1$ of the (real)
    fundamental class~$[M]_\R \in H_d(M;\R)$ of~$M$.

    The $\ell^1$-semi-norm on path-connected spaces also admits other
    geometric descriptions: It is equivalent (in the sense of
    semi-norms) to the volume entropy semi-norm~\cite{babenkosabourau}
    and to the semi-norm generated by URC-manifolds
    (Example~\ref{exa:URC}).
\end{example}

\begin{example}[non-universality of the \lone-semi-norm]\label{exa:nonuniversal}
  For each~$d \in \{3\}\cup\N_{\geq 5}$ there exists a finite functorial
  semi-norm on~$H_d(\;\cdot\;;\R)$ that is \emph{not} carried by the
  $\ell^1$-semi-norm~\cite[Theorem~1.2]{fauserloeh19}. The case~$d = 4$ is
  still wide open at this point.
  On the other hand, all finite functorial semi-norms that are multiplicative
  under finite coverings \emph{are} carried by the
  $\ell^1$-semi-norm~\cite[Proposition~7.11]{crowleyloeh15}.
\end{example}

\begin{example}[singular homology in degrees~$0$ and~$1$]\label{exa:deg01}
  A direct computation shows that for every topological space~$X$ and
  every~$\alpha \in H_0(X;\R)$ with~$\alpha \neq 0$, we
  have~$\|\alpha\|_1 \neq 0$. In particular, $\|\cdot\|_1$
  is a universal finite functorial semi-norm on~$H_0(\args;\R)$. 

  In contrast, every finite functorial semi-norm~$|\cdot|$
  on~$H_1(\args;\R)$ vanishes: We first consider the circle. Because
  $S^1$ admits a self-map~$f \colon S^1 \to S^1$ of degree~$2$,
  functoriality gives the estimate~$2 \cdot \bigl|[S^1]_\R \bigr| =
  \bigl| H_1(f;\R)[S^1]_\R\bigr| \leq \bigl|[S^1]_\R\bigr|$ and thus
  \[ \bigl| [S^1]_\R\bigr| = 0.
  \]
  For the general case, we observe that the Hurewicz theorem and the
  universal coefficient theorem show that for every topological
  space~$X$ and every~$\alpha \in H_1(X;\R)$, there exists~$N \in \N$, 
  continuous maps~$f_1, \dots, f_N \in \map(S^1,X)$, and
  $b_1, \dots, b_N \in \R$ with
  \[ \alpha = \sum_{j=1}^N b_j * H_1(f_j;\R)([S^1]_\R).
  \]
  Therefore, functoriality and the triangle inequality lead to
  $|\alpha| \leq \sum_{j=1}^N |b_j| * \bigl|[S^1]_\R\bigr|
  = 0,
  $ 
  as claimed. In particular, the $\ell^1$-semi-norm is also
  universal on~$H_1(\args;\R)$. The principle of representing
  homology classes by special classes will be discussed in more
  detail in Section~\ref{subsec:genfsn}.
\end{example}

\begin{example}[representable and countably additive functors]
    In the situation of Setup~\ref{setup:F}, if $K \in \{\Q,\R\}$
    and if the functor~$F$ is representable or countably additive,
    then the trivial functorial semi-norm on~$F$ is
    universal~\cite[Corollaries 4.1 and~4.5]{loeh16}.
\end{example}

\subsection{Generating functorial semi-norms}\label{subsec:genfsn}

Functorial semi-norms on singular homology lead to estimates for
mapping degrees; conversely, properties of mapping degrees can be used
to generate functorial semi-norms on singular
homology~\cite[Section~4]{crowleyloeh15}. This way of ``generating
functorial semi-norms via special spaces'' generalises as follows:

\begin{defi}[generated semi-norm]\label{def:generatedsn}
    In the situation of Setup~\ref{setup:F},
    let $S$~be a class of $F$-elements and
    let $v\colon S\to\Rnn\cup\{\infty\}$.
    \begin{itemize}
    \item An \emph{$S$-representation} of an $F$-element $(X,\alpha)$
      is a representation of the form
      \[ \alpha = \sum_{j=1}^N b_j * F(f_j)(\alpha_j)
      \]
      with~$N \in \N$,
      coefficients $b_1, \dots, b_N \in K$,
      $F$-elements $(X_1, \alpha_1), \dots, (X_N, \alpha_N) \in S$, and
      morphisms~$f_1 \colon X_1 \to X, \dots, f_N \colon X_N \to X$ in~$C$.
    \item The \emph{semi-norm~$|\cdot|_v$ on~$F$ generated by~$v$}
      is defined by: For all $F$-elements $\alpha$, we set
      \[ |\alpha|_v
      := \inf\Bigl\{ \sum_{j=1}^N |b_j| * v(\alpha_j)
      \Bigm| \text{$\sum_{j=1}^N b_j * F(f_j)(\alpha_j)$ is an $S$-representation of~$\alpha$}
      \Bigr\},
      \]
      with $\inf\emptyset := \infty$.
    \end{itemize}
\end{defi}

\begin{prop}[generating functorial semi-norms via functions]\label{prop:gensn}
    In the situation of Definition~\ref{def:generatedsn}, we have:
    \begin{enumerate}
        \item
            The semi-norm~$|\cdot|_v$ generated by~$v$ is a functorial
            semi-norm on~$F$.
            
        \item\label{prop:gensn:upperbound}
            For all $F$-elements~$\alpha$ in~$S$, we have
            $|\alpha|_v \leq v(\alpha)$.
            
        \item\label{prop:gensn:monotonic}
            If $v'\colon S\to\Rnn\cup\{\infty\}$ is a function with
            $v' \geq v$ (pointwise), then $|\alpha|_{v'} \geq
            |\alpha|_v$ for all $F$-elements~$\alpha$. 
            In particular, $|\cdot|_{v'}$ carries $|\cdot|_v$.
            
        \item\label{prop:gensn:finite}
            If $S$ contains all $F$-elements given by a skeleton
            of~$C$ and $v$ does not attain~$\infty$, then
            $|\cdot|_v$ is finite.
            
        \item\label{prop:gensn:roundtrip}
            Let $\sigma$ be a functorial semi-norm on~$F$ and
            let $v \geq |\cdot|_\sigma$ on~$S$.
            Then, for all $F$-el\-e\-ments~$\alpha$, we have
            $|\alpha|_v \geq |\alpha|_\sigma$.
    \end{enumerate}
\end{prop}
\begin{proof}
    Using functoriality of~$F$, it is easy to see that $|\cdot|_v$ is
    a functorial semi-norm.
    Also~\pref{prop:gensn:monotonic} follows immediately from the
    definition.
    For an $F$-element $(X,\alpha)$, the identity morphism~$X\to X$
    shows~\pref{prop:gensn:upperbound}.
    Property~\pref{prop:gensn:finite} is a direct consequence
    of~\pref{prop:gensn:upperbound} and the fact that a
    \emph{functorial} semi-norm is uniquely determined by its
    restriction to a skeleton.
    We now prove~\pref{prop:gensn:roundtrip}:
    Let $\sum_{j=1}^N b_j * F(f_j)(\alpha_j) = \alpha$
    be an $S$-representation of~$\alpha$. Then
    \[ |\alpha|_\sigma \leq \sum_{j=1}^N |b_j| * \bigl|F(f_j)(\alpha_j)\bigr|_\sigma
       \leq \sum_{j=1}^N |b_j| * |\alpha_j|_\sigma
       \leq \sum_{j=1}^N |b_j| * v(\alpha_j)
    \]
    by the triangle inequality, functoriality of~$\sigma$, and the
    assumption on~$v$. Taking the infimum over all $S$-representations
    of~$\alpha$, we obtain $|\alpha|_v \geq |\alpha|_\sigma$.
\end{proof}

\begin{remark}[finiteness of generated semi-norms]
    Proposition~\subref{prop:gensn}{:finite}
    only provides a sufficient criterion for~$|\cdot|_v$ to be finite.
    For example, let $d \in \N$ and let us consider the case
    $F = H_d(\;\cdot\;;\R)\colon\!\Top\to\Vect_\R$. Then,
    $|\cdot|_v$ is finite whenever $S$ contains
    and $v$ is finite on enough fundamental classes of manifolds,
    because rational homology classes can (up to
    multiplicity) be realised as the push-forward of fundamental
    classes by a classical result by Thom~\cite{thom54}%
    \cite[Corollary~3.2]{crowleyloeh15}.
    Notably, it is already enough to take the finite
    coverings of a single URC-manifold in dimension~$d$
    (Example~\ref{exa:URC}).
\end{remark}

\begin{example}[semi-norms generated by URC-manifolds]\label{exa:URC}
  Let $d\in \N$. An oriented closed connected $d$-manifold~$M$ is a
  \emph{URC-manifold} (universal realisation of
  cycles)~\cite[p.~1747]{gaifullin13} if for every topological
  space~$X$ and every~$\alpha \in H_d(X;\Z)$, there exists a
  finite-sheeted covering~$\overline M$ of~$M$, a map~$f \in
  \map(\overline M,X)$, and~$b \in \Z \setminus \{0\}$ with
  \[ H_d(f;\Z)\bigl([\overline M]_\Z\bigr) = b * \alpha.
  \]
  For example, the point is a URC-manifold in dimension~$0$, the
  circle is a URC-manifold in dimension~$1$, and oriented closed
  connected surfaces of genus at least~$2$ are URC-manifolds in
  dimension~$2$. Gaifullin showed that (aspherical) URC-manifolds
  exist in every dimension~\cite[Theorem~1.3]{gaifullin13}.

  If $M$ is a URC-manifold in dimension~$d$ and $S$ is the class of
  fundamental classes of all connected finite-sheeted covering
  manifolds of~$M$, then every homology class in~$H_d(\args;\R)$
  admits an $S$-representation. Thus, each function~$v \colon S \to
  \R_{\geq 0}$ generates a finite functorial semi-norm
  on~$H_d(\args;\R)$~\cite[Example~7.10]{crowleyloeh15}.
  
  If $v$ is given by the covering degree, then $|\cdot|_v$ is
  equivalent to the $\ell^1$-semi-norm
  on~$H_d(\args;\R)$~\cite[Theorem~6.1]{gaifullin_comb}.
\end{example}

\section{Universality under equivalence of categories}\label{sec:equiv}

Universal finite functorial semi-norms are compatible with equivalences of
categories (Corollary~\ref{cor:equivalenceofcats}). Indeed, a stronger
result holds: In Proposition~\ref{prop:natretraction}, we show that
universal functorial semi-norms can be transferred along
``weak retractions'' of categories.

\begin{setup} \label{setup:equiv}
    Let $C$ and~$D$ be categories, let $K$~be a normed field and let
    $F\colon C\to \Vect_K$ and $G\colon D\to\Vect_K$ be functors.
    Let $A\colon C\to D$ be a functor such that
    $G\after A$ is naturally isomorphic to~$F$.
    \[
    \begin{tikzcd}[column sep=tiny]
        C \ar[rr, "A", yshift=2pt] \ar[dr, "F"'] & & D \ar[dl, "G"]
        \ar[ll, "B", gray, yshift=-2pt]
        \\
        & \Vect_K &
    \end{tikzcd}
    \]
\end{setup}

\begin{prop} \label{prop:natretraction}
    In the situation of Setup~\ref{setup:equiv}, let $B\colon D\to C$
    be a right-inverse of~$A$, i.e., we assume that $A\after B$ is
    naturally isomorphic to the identity on~$D$.
    Then, if $F$ admits a universal functorial semi-norm, so does~$G$.
\end{prop}

As an immediate consequence, we obtain:

\begin{cor} \label{cor:equivalenceofcats}
    In the situation of Setup~\ref{setup:equiv},
    assume that $A\colon C\to D$ is an equivalence of categories.
    Then $F$~admits a universal finite functorial semi-norm if and
    only if $G$~does.
\end{cor}

Before we give the proof of Proposition~\ref{prop:natretraction}, we
make a few remarks about the interplay between functorial semi-norms
and natural isomorphisms:

\begin{remark}[non-strict functorial semi-norms]\label{rem:nonstrictfsn}
    In the situation of Setup~\ref{setup:equiv} and
    given a functorial semi-norm~$\tau$ on~$G$, one would like to
    precompose $\tau$ with~$A$ to get a functorial semi-norm on~$F$.
    However, as $G\after A$ is not necessarily \emph{equal} to~$F$,
    also $\tau\after A$ will not necessarily be a strict lift of~$F$,
    but only up to natural isomorphism. In other words: if
    $U\colon\!\snVect_K\to\Vect_K$ denotes the forgetful functor,
    the right triangle in the diagram
    \[
    \begin{tikzcd}[column sep=tiny]
        C \ar[rr, "A", yshift=2pt] \ar[dr, "F"']
            \ar[ddr, dashed, "\tau\after A"', bend right]
        & &
        D \ar[dl, "G"] \ar[ddl, "\tau", bend left]
        \\
        & \Vect_K &
        \\
        & \snVect_K \ar[u, "U"] &
    \end{tikzcd}
    \]
    commutes on the nose while the other two only commute up to
    natural isomorphism.
    
    One possible way to proceed would be to relax the definition of
    functorial semi-norm: Instead of $U\after\tau = G$ we only
    require~$U\after\tau \isom G$, and then the functorial
    semi-norm consists of $\tau$~together with such a natural
    isomorphism. 
  
    This sounds like the correct setting to pursue the categorical view
    on functorial semi-norms (or formalisation in a proof
    assistant~\cite[Chapter~4.1.2]{loeh_explore}).
    On the other hand, this setting does not actually increase the
    pool of functorial semi-norms:
    Indeed, if $\eta \colon G \natTo U \after \tau$ is a natural
    isomorphism, the technique from Remark~\ref{rem:pullback} will show
    how to construct a (strict) functorial semi-norm on~$G$ ``with the
    same semi-norms''.
\end{remark}

\begin{remark}[pull-back along natural transformation]
    \label{rem:pullback}%
    Let $C$ be a category, let $K$ be a normed field,
    let $\eta\colon F\natTo F'$ be a natural transformation of
    functors~$C\to\Vect_K$, and
    let $\sigma$~be a functorial semi-norm on~$F'$.
    Then, by naturality of~$\eta$,
    \[ C\to\snVect_K, \quad
    \begin{cases}
        X \mapsto
            \bigl( F(X), (\eta_X)^\ast|\cdot|_{\sigma} \bigr)
        & \text{on objects}
        \\
        \hphantom{X}\mathllap{f} \mapsto F(f)
        & \text{on morphisms}
    \end{cases}
    \]
    defines a functorial semi-norm~$\eta^\ast \sigma$ on~$F$.
\end{remark}

\phantomsection%
\belowpdfbookmark{Proof of Proposition~\ref{prop:natretraction}}%
                 {proof of prop:natretraction}
\begin{proof}[Proof of Proposition~\ref{prop:natretraction}]
    First, we fix some notation:
    Let $\sigma$~be a universal finite functorial semi-norm on~$F$.
    Let $\lambda \colon\! \Id_D \natTo A \after B$
    and $\psi \colon F \natTo G \after A$ be natural
    isomorphisms. Then~$\varphi := \psi^{-1} \circ G(\lambda)$ is a
    natural isomorphism~$G \natTo F \after B$.  We consider the
    induced functorial semi-norm~$\tilde\sigma := \varphi^\ast (\sigma
    \after B)$ on~$G$ (Remark~\ref{rem:pullback}).
        
    We show that $\tilde\sigma$ is universal for~$G$: Let $\tau$ be a
    finite functorial semi-norm on~$G$.  The idea is straightforward:
    We go to the side of~$F$, compare the result with the
    universal~$\sigma$ on~$F$, and then derive universality
    of~$\tilde\sigma$ on~$G$. However, this involves a round-trip from
    $D$ over $C$ back to~$D$, and thus we have to take~$\lambda$ into
    account. More precisely, we proceed as follows:
    \begin{enumerate}
        \item
            Let $(Y,\tilde\beta)$ be a $G$-element
            with $|\tilde\beta|_{\tilde\sigma} = 0$. We need to show
            that we also have $|\tilde\beta|_\tau = 0$.
            
        \item
            In order to prepare for the round-trip,
            we twist~$\tau$ by~$\lambda$ and obtain
            the finite functorial semi-norm~$\tau_\lambda :=
            G(\lambda)^\ast(\tau \after A \after B)$ on~$G$.
            
        \item
            Using~$A$ and~$\psi$, we can pull this back to the finite functorial
            semi-norm~$\tilde\tau := \psi^\ast (\tau_\lambda \after A)$
            on~$F$, which we can now relate to~$\sigma$.

            Let $\beta := \varphi_Y(\tilde\beta) \in F(B(Y))$
            be the element corresponding to~$\tilde\beta$.
            By construction, we have
            \[ |\beta|_\sigma
            = \bigl| \varphi_Y(\tilde\beta)\bigr|_\sigma
            = \bigl| \varphi_Y(\tilde\beta)\bigr|_{\sigma\after B}
            = |\tilde\beta|_{\varphi^\ast(\sigma \after B)}
            = |\tilde\beta|_{\tilde\sigma}
            = 0;
            \]
            in the second step, we reinterpreted $\varphi_Y(\tilde\beta)$
            as element of~$F \circ B(Y)$, so that instead of~$\sigma$
            on~$B(Y)$ we can equivalently apply~$\sigma \after B$
            on~$Y$.
            
            From universality of~$\sigma$, we hence obtain
            $|\beta|_{\tilde\tau} = 0$.
            
         \item
            In the last step, we translate this result back to~$\tau$.
            To keep the notation light, we will not explicitly annotate
            the objects to which the natural transformations are applied. 
            We compute
            \begin{align*}
              0
              &
              = |\beta|_{\tilde \tau} 
              = |\beta|_{\psi^\ast(\tau_\lambda\after A)} 
              = \bigl|\psi(\beta)\bigr|_{\tau_\lambda \circ A} 
              = \bigl|\psi(\beta)\bigr|_{\tau_\lambda} 
              = \bigl|\psi(\beta)\bigr|_{G(\lambda)^\ast(\tau\after A\after B)} 
              \\
              &
              = \bigl| G(\lambda) (\psi(\beta))
                \bigr|_{\tau \after A \after B} 
              = \bigl| G(\lambda) (\psi(\beta))
                \bigr|_{\tau}. 
            \end{align*}
            For every object~$Z$ of~$D$,
            the map~$G(\lambda_Z)$ is an isometry with respect to~$|\cdot|_\tau$
            because $\lambda_Z$ is an isomorphism in~$D$ and $\tau$ is a
            functorial semi-norm on~$G$. Therefore, we can continue with 
            \begin{align*}
                \bigl| G(\lambda) (\psi(\beta))
                \bigr|_{\tau} 
              &
              = \bigl| \psi(\beta)
                \bigr|_\tau 
              = \bigl| \psi (\varphi (\tilde\beta))
                \bigr|_\tau 
              = \bigl| \psi \after \psi^{-1} \after G(\lambda) (\tilde \beta)
                \bigr|_\tau 
              = \bigl| G(\lambda)(\tilde \beta)\bigr|_\tau 
              \\
              &
              = |\tilde\beta|_\tau.
            \end{align*}
            We conclude that $|\tilde\beta|_\tau =0$, as claimed.
            \qedhere
    \end{enumerate}
\end{proof}

\section{Vanishing loci}\label{sec:vanishingloci}

In this section, we reformulate the ``carries'' relation
(Definition~\ref{def:carries}) in terms of vanishing loci
(Definition~\ref{def:locus}, Remark~\ref{rem:Ncarry}).

The vanishing loci provide a convenient language to reason about
families of functorial semi-norms and their relations: In
Section~\ref{subsec:carryseq}, we use a diagonalisation construction
on the associated functions to construct a functorial semi-norm that
carries countably many given functorial semi-norms
(Proposition~\ref{prop:carrySeqGenerated} and
Corollary~\ref{cor:carrySeqSn}).

\begin{setup} \label{setup:locus}%
    Let $C$ be a category, let $K$~be a normed field, let
    $F\colon C\to\Vect_K$ be a functor, and let $S$~be a class
    of $F$-elements.
\end{setup}

\subsection{A reformulation of carrying}

\begin{defi}[vanishing locus] \label{def:locus}
  We assume Setup~\ref{setup:locus}; let $X \in \Ob(C)$.
  \begin{itemize}
  \item For a functorial semi-norm~$\sigma$ on~$F$, we define
    the \emph{vanishing locus} of~$\sigma$ on~$X$ by 
    \[ N_\sigma(X)
     := \bigl\{ \alpha \in F(X) \bigm| |\alpha|_\sigma = 0 \bigr\}.
    \]
  \item If $C$~is small, we write $\Fsn(F)$ for the class of all
    finite functorial semi-norms on~$F$ and set
      \[ N(X) := \bigcap_{\sigma \in \Fsn(F)} N_\sigma(X).
      \]
  \item For a function $v\colon S \to \Rnn$,
        we write $N_v(X)$ for $N_{|\cdot|_v}(X)$,
        where $|\cdot|_v$ is the functorial semi-norm generated
        by~$v$ (Proposition~\ref{prop:gensn}).
  \end{itemize}
\end{defi}

In the situation of the definition, $N_\sigma(X)$ and $N(X)$
are $K$-subspaces of~$F(X)$ and $N(X) \subset N_\sigma(X)$.
Furthermore, if we regard $\Fsn(F)$ as the preorder category with
respect to the ``carries'' relation, an initial object of this
category is precisely a universal finite functorial semi-norm on~$F$,
while the trivial functorial semi-norm is always a terminal one.

\begin{remark}\label{rem:Ncarry}
    In the situation of Setup~\ref{setup:locus}, let
    $\sigma$,~$\tau$ be functorial semi-norms on~$F$.
    \begin{enumerate}
        \item
            Then $\sigma$ carries~$\tau$ if and only if
            \[ \fa{X \in \Ob(C)} N_\sigma(X) \subset N_\tau(X) .
            \]
        \item\label{rem:Ncarry:universal}
            If $C$ is small, $\sigma$~is universal on~$F$ if and only
            if it is finite and fulfills
            \[ \fa{X \in \Ob(C)} N_\sigma(X) \subset N(X).
            \]
        \item\label{rem:Ncarry:function}
            By Proposition~\subref{prop:gensn}{:roundtrip},
            the functorial semi-norm generated by
            $S\to\Rnn,\, \alpha\mapsto |\alpha|_\sigma$ carries~$\sigma$, i.e.,
            \[ \fa{X \in \Ob(C)}
                N_{\alpha\mapsto |\alpha|_\sigma}(X) \subset N_\sigma(X) .
            \]
    \end{enumerate}
\end{remark}

\subsection{Carrying a sequence of semi-norms}
\label{subsec:carryseq}

The main ingredient for the proof of Theorem~\ref{thm:genuffsn}
is that we can simultaneously carry countably many finite functorial
semi-norms generated on a countable class of elements:

\begin{prop}
  \label{prop:carrySeqGenerated}
  In the situation of Setup~\ref{setup:locus},
  let $S$ be countable and
  let $(v_n)_{n \in \N}$ be a sequence of
  functions~$S \to \R_{\geq 0}$. 
  Then there exists a function~$v \colon S \to \R_{\geq 0}$
  such that $|\cdot|_v$ carries all~$(|\cdot|_{v_n})_{n\in \N}$,
  i.e., with
  \[ \fa{X \in \Ob(C)}
     N_v(X) \subset \bigcap_{n \in \N} N_{v_n}(X).
  \]
  In particular: If $C$ is small, if every $F$-element admits an
  $S$-representation, and if
  \[ \fa{X \in \Ob(C)}
     \bigcap_{n \in \N} N_{v_n}(X) \subset N(X), 
  \]
  then $|\cdot|_v$ is universal for~$F$.
\end{prop}
\begin{proof}
  The second part follows from the first part and the
  characterisation of universality from Remark~\subref{rem:Ncarry}{:universal}.

    We now prove the first part.
    As indicated by Proposition~\subref{prop:gensn}{:monotonic}, we
    would like to set ``$v := \sup_n v_n$'', but of course
    this might not produce a \emph{finite valued} function. So instead, we
    choose 
    \[ v\colon S\to\Rnn, \;
        \alpha_n \mapsto \max\bigl\{ v_j(\alpha_n) \bigm| j\in\{-1,\dots,n\} \bigr\}
        ,
    \]
    where we fix and implicitly use an enumeration
    $(X_n,\alpha_n)_{n\in\N}$ of~$S$ and where $v_{-1} := 1$.
    
    \begingroup%
    \newcommand{\kj}{{k_j}}%
    \newcommand{\vm}{{v_m}}%
    In order to show that $v$~has the claimed property, we let
    $m\in\N$ and show that $|\cdot|_v$ carries~$|\cdot|_{v_m}$: 
    We introduce the following constants: Let $q_{-1} := 1$, let
    \[
        q_k :=
        \begin{cases}
            v(\alpha_k) * |\alpha_k|_\vm\inv &
                \text{if } |\alpha_k|_\vm > 0,
            \\
            1 & \text{if } |\alpha_k|_\vm = 0
        \end{cases}
    \]
    for all $k\in\setZeroTo{m}$, and let
    \[
        Q := \min\bigl\{ q_k \bigm| k\in\{-1,\dots,m\} \bigr\}
    .
    \]
    By construction, we have that~$Q \in (0,1]$. 
    For every $F$-element~$\alpha$ and every
    $S$-rep\-re\-sen\-ta\-tion~$\alpha = \sum_{j=1}^N b_j * F(f_j)(\alpha_\kj)$,
    we can estimate
    \begingroup%
        \newcommand{\sumL}{ \sum_{\substack{j\in\setOneTo{N} \\ k_j < m}}}%
        \newcommand{\sumGE}{\sum_{\substack{j\in\setOneTo{N} \\ k_j \geq m}}}%
    \begin{align*}
        \sum_{j=1}^N |b_j| * v(\alpha_\kj)
        &\geq \sumL |b_j| * q_\kj * |\alpha_\kj|_\vm
        + \sumGE |b_j| * v_m(\alpha_\kj)
        & \text{(definition of~$q_\kj$ and $v$)}
        \\
        &\geq Q * \sumL |b_j| * |\alpha_\kj|_\vm
        + \sumGE |b_j| * |\alpha_\kj|_\vm
        & \text{(def.\ of $Q$ and~P.~\subref{prop:gensn}{:upperbound})}
        \\
        &\geq Q * \sum_{j=1}^N |b_j| * |\alpha_\kj|_\vm
        & \text{(because~$Q \leq 1$)}
        \\
        & \geq Q * |\alpha|_\vm
    , \end{align*}
    \endgroup%
    where the last step follows from applying~$|\cdot|_\vm$ to the given
    $S$-representation of~$\alpha$.
    By taking the infimum over all such $S$-representations, we obtain
    $|\alpha|_v \geq Q * |\alpha|_\vm$. As~$Q > 0$, we see that  
    $|\cdot|_v$ carries $|\cdot|_\vm$ as desired.%
    \endgroup
\end{proof}

\begin{cor} \label{cor:carrySeqSn}
  In the situation of Setup~\ref{setup:locus},
  let $C$ be small,
  let $S$ be countable, and
  let $T\subset\Fsn(F)$ be countable.
  Then there exists a functorial semi-norm~$\sigma$ on~$F$
  such that $\sigma$ carries all of~$T$,
  i.e., with
  \[ \fa{X \in \Ob(C)}
    N_\sigma(X) \subset \bigcap_{\tau \in T} N_\tau(X).
  \]
  In particular: If every $F$-element admits an $S$-representation
  and if
  \[ \fa{X \in \Ob(C)}
     \bigcap_{\tau \in T} N_\tau(X) \subset N(X),
  \]
  then $\sigma$ is universal for~$F$.
\end{cor}
\begin{proof}
    Again, the second part follows from the first one and
    Remark~\subref{rem:Ncarry}{:universal}.
    
    We prove the first part of the claim:
    By Remark~\subref{rem:Ncarry}{:function},
    for each $\tau \in T$, we find a function
    $v_\tau \colon S\to\Rnn$ with
    \[ \fa{X \in \Ob(C)}
        N_{v_\tau}(X) \subset N_\tau(X)
    . \]
    We then choose an enumeration of~$\{ v_\tau \mid \tau \in T \}$
    and apply Proposition~\ref{prop:carrySeqGenerated}.
\end{proof}

\section{Existence of universal finite functorial semi-norms}\label{sec:exproof}

In this section, we prove Theorem~\ref{thm:genuffsn} and
Corollary~\ref{cor:singhomfin} on singular homology. We first treat
the case of countable fields where a direct enumeration argument
applies (Section~\ref{subsec:uffsn:countable}).  In
Section~\ref{subsec:uffsn:finrange}, we consider functors with range
in finite dimensional vector spaces over general normed fields.

In both cases, we use the following observation:

\begin{remark} \label{rem:objectsCountable}
    By definition, the inclusion functor of a skeleton into the
    ambient category is an equivalence. Invoking
    Corollary~\ref{cor:equivalenceofcats}, we may equivalently
    assume that the category itself has only countably many objects.
\end{remark}

\subsection{The countable case} \label{subsec:uffsn:countable}

\phantomsection%
\belowpdfbookmark{Proof of Theorem~\subref{thm:genuffsn}{:countable}}%
                 {proof of main theorem part one}
\begin{proof}[Proof of Theorem~\subref{thm:genuffsn}{:countable}]
    We may assume that $C$~itself has only countably many objects
    (Remark~\ref{rem:objectsCountable}).
    Furthermore, by assumption, $K$~and $\dim_K F(X)$ are countable
    for all objects~$X$ of~$C$.
    Together, we obtain that the class~$S$ of \emph{all} $F$-elements
    is a countable set. Trivially, all $F$-elements admit an
    $S$-representation.
    
    Let
    $S' := \{ (X,\alpha) \in S \mid \alpha\notin N(X) \}$
    and for each $(X,\alpha) \in S'$ let
    $\sigma_\alpha$ be a finite functorial semi-norm on~$F$
    with $\alpha\notin N_{\sigma_\alpha}(X)$.
    
    By construction, for every object~$Y$ of~$C$, we have
    \[ F(Y) \setminus N(Y) \subset
        \bigcup_{(X,\alpha)\in S'} F(Y) \setminus N_{\sigma_\alpha}(Y).
    \]
    Hence, by De~Morgan's laws and Corollary~\ref{cor:carrySeqSn},
    there exists a universal functorial semi-norm on~$F$.
\end{proof}

\begin{remark}
    In general, it would \emph{not} be enough to have a countable
    set~$S$ with the property that every $F$-element admits an
    $S$-representation. Without the countability assumption
    on~$\Ob(C)$, it might not be possible to control the vanishing
    locus on all objects by only countably many functorial semi-norms,
    and thus, the second part of Corollary~\ref{cor:carrySeqSn} does
    not apply. A concrete example is given in Section~\ref{sec:nonuffsn}.
\end{remark}

\subsection{The case of finite dimensional range}
\label{subsec:proofsinghomfin}\label{subsec:uffsn:finrange}

We prove the second part of Theorem~\ref{thm:genuffsn}
and derive Corollary~\ref{cor:singhomfin}. 
As a preparation, we show that we can achieve
universality on a single object:

\begin{lemma}
  \label{lem:uffsnX}
  Let $C$ be a small category, let $K$ be a normed field,
  and let $F \colon C \to \Vect_K$ be a functor.
  Let $X \in \Ob(C)$ with~$\dim_K F(X) < \infty$.
  Then there exists a finite functorial semi-norm~$\sigma$
  on~$F$ with
  \[ N_\sigma(X) = N(X).
  \]
\end{lemma}
\begin{proof}
  We proceed inductively, using the following
  observation:
  \begin{itemize}
  \item[]
  If $\sigma \in \Fsn(F)$ with~$N_\sigma(X) \neq N(X)$,
  then there exists a~$\sigma' \in \Fsn(F)$ with
  \[ \dim_K N_{\sigma'}(X) < \dim_K N_\sigma(X).
  \]
  \end{itemize}
  Indeed, if $N_\sigma(X) \neq N(X)$, 
  there exists an~$\alpha \in N_\sigma(X) \setminus N(X)$.
  Hence, there is a finite functorial semi-norm~$\tau$
  on~$F$ with~$|\alpha|_\tau \neq 0$.
  Then also $\sigma' := \sigma + \tau \in \Fsn(F)$ and
  $\alpha$ witnesses that
  \[ N_{\sigma'}(X) \subset N_\sigma(X) \cap N_\tau(X)
     \subsetneq N_\sigma(X).
  \]
  Because of~$\dim_K N_\sigma(X) \leq \dim_K F(X) < \infty$,
  we obtain~$\dim_K N_{\sigma'}(X) < \dim_K N_\sigma(X)$.

  For the actual induction, we start with the trivial
  functorial semi-norm~$\sigma := 0$ on~$F$, which
  satisfies~$N_\sigma(X) = F(X)$. We then iteratedly
  apply the observation above. Because $\dim_K F(X)$
  is finite, this will terminate and lead to a finite
  functorial semi-norm~$\sigma$ on~$F$ with~$N_\sigma(X)
  = N(X)$.
\end{proof}

\phantomsection%
\belowpdfbookmark{Proof of Theorem~\subref{thm:genuffsn}{:fin}}%
                 {proof of main theorem part two}
\begin{proof}[Proof of Theorem~\subref{thm:genuffsn}{:fin}]
    By Remark~\ref{rem:objectsCountable}, we may assume without loss
    of generality, that $\Ob(C)$ is countable.
  For each~$X\in\Ob(C)$,
  let $(\alpha_i)_{i \in I_X}$ be a finite generating set
  of the finite-dimensional $K$-vector space~$F(X)$.
  Then $S := \{ (X,\alpha_i) \mid X \in \Ob(C),\ i \in I_X \}$
  is countable and every $F$-element admits an $S$-representation.

  By Lemma~\ref{lem:uffsnX},
  for each~$X \in \Ob(C)$, we find a functorial semi-norm~$\sigma_X$
  on~$F$ with $N_{\sigma_X}(X) = N(X)$.
  Therefore, for all~$Y \in \Ob(C)$, we have
  \[ \bigcap_{X \in \Ob(C)} N_{\sigma_X}(Y) \subset N(Y).
  \]

  Applying Corollary~\ref{cor:carrySeqSn} to the countable
  set~$\{\sigma_X \mid X \in \Ob(C)\}$ thus shows that there exists a
  universal functorial semi-norm on~$F$.
\end{proof}

\phantomsection%
\belowpdfbookmark{Proof of Corollary~\ref{cor:singhomfin}}%
                 {proof of corollary to main theorem}
\begin{proof}[Proof of Corollary~\ref{cor:singhomfin}]
  Let $T$ be the category of all topological spaces
  that are homotopy equivalent to a finite CW-complex;
  as morphisms in~$T$, we take all continuous maps.

  Every functorial semi-norm on~$H_d(\args;K)$ is homotopy
  invariant in the sense that homotopy equivalences induce
  isometric isomorphisms on~$H_d(\args;K)$. Thus, it suffices
  to show that the functor~$F \colon \htpycat{T} \to \Vect_K$
  on the homotopy category~$\htpycat{T}$ of~$T$ induced by~$H_d(\args;K)$
  admits a universal finite functorial semi-norm.

  As there are only countably many homotopy types of finite 
  CW-complexes (Remark~\ref{rem:countingCW}), the
  category~$\htpycat{T}$ has a skeleton with countably many
  objects. Moreover, $\dim_K H_d(X;K) < \infty$ for all finite
  CW-complexes~$X$.

  Therefore, the second part of Theorem~\ref{thm:genuffsn}
  applies and we obtain that $F$ admits a universal finite
  functorial semi-norm.
\end{proof}

\begin{remark}[counting CW-complexes]\label{rem:countingCW}
  A simple counting argument shows that there are only countably many
  homeomorphism types of finite simplicial complexes.  As every finite
  CW-complex is homotopy equivalent to a finite simplicial
  complex~\cite[Theorem~2C.5]{hatcher},
  it follows that there are only countably many homotopy types of
  finite CW-complexes.

  In contrast, there are uncountably many homotopy types of countable
  CW-complexes.  Looking at the fundamental group and presentation
  complexes shows that there are even uncountably many homotopy types
  of countable $2$-dimensional CW-complexes whose $1$-skeleton is~$S^1
  \lor S^1$ (because there are uncountably many isomorphism types of
  $2$-gen\-er\-at\-ed groups).

  From a constructive point of view, an interesting category of topological
  spaces with a skeleton that has only countably many objects is the
  category of recursively enumerable simplicial complexes.
\end{remark}

\begin{remark}[base change]
  In general, it does not seem to be clear how universal functorial
  semi-norms behave under base change. For example, if $\sigma$ is a
  universal finite functorial semi-norm on a functor~$F$
  to~$\QVectfin$, then it is not clear whether $\R \otimes_\Q \sigma$,
  defined by the object-wise tensor product with the standard norm
  on~$\R$, is universal for~$\R \otimes_\Q F$. Indeed, it is a priori
  not clear how the vanishing loci transform under such base changes.
\end{remark}

\begin{remark}[universal finite functorial semi-norms generated by
    URC-manifolds]\label{rem:URCgrowth}
  Let $d \in \N$, let $M$ be a URC-manifold, and let $S$ be the class
  of fundamental classes of all connected finite-sheeted covering
  manifolds of~$M$ (Example~\ref{exa:URC}).  If $d \geq 2$, then for
  each~$(X,[X]_{\R})$ all covering maps~$X \to M$ have the same
  number of sheets (this can be derived using simplicial volume); we
  denote this number by~$k(X)$. For every~$k \in \N$, there are only
  finitely many homeomorphism types~$S_k$ of~$(X,[X]_\R) \in S$
  with~$k(X) = k$, as can be seen from the classification of coverings
  and the fact that the finitely generated group~$\pi_1(M)$ contains
  only a finite number of subgroups of index~$k$.
  
  Let $v \colon S \to \R_{\geq 0}$. We can thus 
  define the modified function
  \begin{align*}
    \overline v \colon S
    & \to \R_{\geq 0},
    \\
    (X,\alpha)
    & \mapsto \max_{(X',[X']_\R) \in S_{k(X)}} \bigl|[X']_\R\bigr|_v.
  \end{align*}
  By construction~$\overline v \geq v$ and so $|\cdot|_{v}$ is carried
  by~$|\cdot|_{\overline v}$
  (Proposition~\subref{prop:gensn}{:monotonic}).
  
  Hence, Question~\ref{q:explicit} can be reformulated as follows: How
  fast does~$\overline v$ have to grow in the covering degree to
  ensure that $|\cdot|_{\overline v}$ is a universal finite functorial
  semi-norm on~$H_d(\args;\R)$ on the category of topological spaces
  homotopy equivalent to finite CW-complexes?  In view of
  Example~\ref{exa:nonuniversal} and Example~\ref{exa:URC}, we know
  that for~$d \in \{3\} \cup \N_{\geq 5}$, the growth for universal
  examples must be faster than linear.
\end{remark}

\section{Situations without universal finite functorial semi-norms}\label{sec:nonuffsn}

We give an example of a functor to~$\QVectfin$ that does not admit a
universal finite functorial semi-norm
(Proposition~\ref{prop:nonuffsn}). In accordance with
Theorem~\ref{thm:genuffsn}, the domain category will not admit a
skeleton with countably many objects. 

\begin{defi}[the category~$C$]\label{def:exC}
  We define a category~$C$ by:
  \begin{itemize}
  \item 
    We set~$M := (\R_{\geq 1})^\N$ and
    $\Ob(C) := \N \sqcup M.$
  \item
    The only non-identity morphisms in~$C$ are the
    morphisms~$f_{m,v} \colon m \to v$ with~$m \in \N$ and
    $v \in M$.
  \end{itemize}  
\end{defi}
  
\begin{defi}[the functor~$F$]\label{def:exF}
  We define a functor~$F \colon C \to \QVectfin$ as follows:
  \begin{itemize}
  \item For all objects~$X \in \Ob(C)$, we set~$F(X) := \Q$.
  \item For~$m \in \N$ and $v \in M$, we set
    \[ F(f_{m,v}) := d_{m,v} \cdot \id_\Q,
    \]
    where $d_{m,v} := \lceil v(m) \rceil$. 
  \end{itemize}
\end{defi}

We will show that $F \colon C \to \QVectfin$ does not admit a
universal finite functorial semi-norm. To this end we use the
following class to generate functorial semi-norms in the sense of
Proposition~\ref{prop:gensn}:

\begin{defi}[the class~$S$]\label{def:exS}
    For clarity, we denote by $\oneInF{X}$ the element $1\in\Q = F(X)$
    for every object $X\in\Ob(C)$.
    We define \[ S := \bigl\{ (m,\oneInF{m}) \mid m\in\N \bigr\} \]
    and for a function $v\colon\N\to\R_{\geq 0}$, we
    write~$|\cdot|_v := |\cdot|_{(m,\oneInF{m})\mapsto v(m)}$
    for the generated functorial semi-norm on~$F$.
\end{defi}

First we show, that we understand $S$-representations well enough
to compute the generated semi-norms on~$F$:

\begin{lemma}\label{lem:exSassoc}
  In the situation of Definition~\ref{def:exS}, for all~$v\colon \N
  \to \R_{\geq0}$ and $w\colon \N \to \R_{\geq 1}$, we have
  \[ |\oneInF{w}|_v
    = \inf_{m \in \N} \frac1{d_{m,w}} \cdot v(m).
  \]
\end{lemma}
\begin{proof}
  The $S$-representations~$1_w = \frac{1}{d_{m,w}} *
  F(f_{m,w})(1_m)$ for $m\in\N$ show that ``$\leq$'' holds.

  Conversely, every $S$-representation of~$\oneInF{w}$ is
  of the form~$\sum_{j=1}^N b_j \cdot F(f_{m_j,w})(\oneInF{m_j})$
  with certain~$b_j \in \Q$ and $m_j \in \N$.
  In particular,
  \[ 1 = |\oneInF{w}|_\Q
  = \Bigl|\sum_{j =1}^N b_j \cdot d_{m_j,w}\Bigr|_\Q
  \leq \sum_{j=1}^N |b_j|_\Q \cdot d_{m_j,w}
  \]
  and so
  \begin{align*}
    \sum_{j=1}^N |b_j|_\Q \cdot v(m_j)
    & \geq \sum_{j=1}^N |b_j|_\Q \cdot d_{m_j,w}
    \cdot \inf_{m \in \N} \frac1{d_{m,w}} \cdot v(m)
    \\
    & \geq 1 \cdot \inf_{m \in \N} \frac1{d_{m,w}} \cdot v(m).
  \end{align*}
  Taking the infimum over all $S$-representations of~$\oneInF{w}$
  finishes the proof.
\end{proof}
  
\begin{prop}\label{prop:nonuffsn}
  Let $F \colon C \to \QVectfin$ be the functor constructed in
  Definition~\ref{def:exF} on the category from
  Definition~\ref{def:exC}.  Then, there is \emph{no} universal finite
  functorial semi-norm on~$F$.
\end{prop}
\begin{proof}
  \emph{Assume} for a contradiction that $F$ admits a universal
  finite functorial semi-norm~$|\cdot|$.
  Let $S$ be the class from Definition~\ref{def:exS} and let 
  \begin{align*}
    v \colon \N & \to \R_{\geq 0} \\
    m & \mapsto | \oneInF{m} |.
  \end{align*}
  Then, $v$ generates a functorial semi-norm~$|\cdot|_v$
  on~$F$ via~$S$ (Proposition~\ref{prop:gensn}, Definition~\ref{def:exS}).

  We now consider the function
  \begin{align*}
    w \colon \N & \to \R_{\geq 1} \\
    m & \mapsto m \cdot v(m) + 1
  \end{align*}
  and its generated finite functorial semi-norm~$|\cdot|_w$ on~$F$.

  We show that $|\cdot|_w$ is \emph{not} carried by~$|\cdot|$: Let
  $\alpha := \oneInF{w}$. On the one hand, by
  Lemma~\ref{lem:exSassoc}, we obtain
  \[
  |\alpha|_w
  = \inf_{m \in \N} \frac1 {d_{m,w}} \cdot w(m)
  = \inf_{m \in \N} \frac{w(m)}{\lceil w(m)\rceil}
  \geq \frac12.
  \]
  On the other hand, we have
  (Proposition~\subref{prop:gensn}{:roundtrip} and
  Lemma~\ref{lem:exSassoc})
  \[
  |\alpha|
  \leq |\alpha|_v
  = \inf_{m \in \N} \frac1 {d_{m,w}} \cdot v(m)
  =    \inf_{m \in \N} \frac{v(m)}{\lceil m \cdot v(m) + 1\rceil}
  \leq \inf_{m \in \N_{>0}} \frac 1m
  =    0.
  \]
  Hence, $\alpha$ witnesses that $|\cdot|_w$ is not carried
  by~$|\cdot|$.
\end{proof}

It does not seem clear whether this phenomenon could be replicated for
the singular homology functor on the category of topological spaces.

{\small
  \bibliographystyle{alpha}
  \bibliography{bib_uffsn}
  }

\vfill

\noindent
\emph{Clara L\"oh,\\ Johannes Witzig}\\[.5em]
  {\small
  \begin{tabular}{@{\qquad}l}
    Fakult\"at f\"ur Mathematik,
    Universit\"at Regensburg,
    93040 Regensburg\\
    \textsf{clara.loeh@mathematik.uni-r.de}, 
    \textsf{https://loeh.app.ur.de}
    \\
    \textsf{johannes.witzig@mathematik.uni-r.de}
  \end{tabular}}

\end{document}